\documentclass[11pt]{amsart}

\usepackage{amsmath,amsthm,amssymb,mathrsfs,amsfonts,verbatim,enumitem,color,leftidx}
\usepackage{mathabx}
\usepackage{etoolbox} 
\usepackage{bbm}
\usepackage[all,tips]{xy}
\usepackage{graphicx,ifpdf}
\usepackage{stmaryrd}
\ifpdf
\fi
\usepackage[colorlinks]{hyperref}
\hypersetup{
linkcolor=blue,          
citecolor=green,        
}

\usepackage{tikz}

\usepackage{marginnote}
\usepackage{color}

\newtheorem{thm}{Theorem}[section]
\newtheorem{lem}[thm]{Lemma}
\newtheorem{coro}[thm]{Corollary}

\newtheorem{defn}[thm]{Definition}
\newtheorem{rmk}[thm]{Remark}

\numberwithin{equation}{section}

\theoremstyle{remark}
\numberwithin{equation}{section}

\definecolor{esperance}{rgb}{0.0,0.5,0.0}

\newcommand{\bb}{\mathbf{b}}

\newcommand{\be}{\mathbf{e}}

\newcommand{\bp}{\mathbf{p}}
\newcommand{\bq}{\mathbf{q}}
\newcommand{\bs}{\mathbf{s}}

\newcommand{\br}{\mathbf{r}}

\newcommand{\eb}{\overline{e}}

\newcommand{\cC}{\mathcal{C}}
\newcommand{\cD}{\mathcal{D}}

\newcommand{\cG}{\mathcal{G}}

\newcommand{\cP}{\mathcal{P}}

\newcommand{\bP}{\mathfrak{P}}

\newcommand{\bZ}{\mathbb{Z}}

\newcommand{\bN}{\mathbb{N}}

\newcommand{\onto}{\xymatrix{\ar@{>>}[r]&}}

\begin{document}

\title{Notes on the values of volume entropy of graphs
}

\author{WooYeon Kim, Seonhee Lim}

\thanks{}



\date{}


\begin{abstract} Volume entropy is an important invariant of metric graphs as well as Riemannian manifolds. In this note, we calculate the change of volume entropy when an edge is added to a metric graph or when a vertex and edges around it are added.
In the second part, we estimate the value of the volume entropy which can be used to suggest an algorithm for calculating the persistent volume entropy of graphs.
\end{abstract}

\maketitle

\section{Introduction}
Let $\cG$ be a finite metric graph with more than two vertices. The volume entropy of the graph $\cG$ is defined as the maximal exponential growth rate of the volume of metric ball $B(\tilde x,r)$ of center $\tilde x$ and radius $r$ in the universal cover of $\cG$:
$$ h = \max_{x \in V\cG}  \lim_{r \to \infty} \frac{\log \mathrm{vol} B(\tilde x,r)}{r},$$
where the volume of $B(\tilde x,r)$ is the sum of length of all the edges (or part of the edges) in $B(\tilde x,r) \subset \widetilde{\cG}$. When $\cG$ is connected, the limit does not depend on the vertex $x$ by subadditivity.

Volume entropy, unlike the usual (measure-theoretic) entropy in dynamical systems, is a metric invariant. It is equal to the topological entropy of the geodesic flow \cite{Man}, \cite{Lim} for non-positively curved Riemannian manifold or piecewise Riemmanian manifolds, as well as for finite metric graphs \cite{Lim}, \cite{Lim2}.
The volume entropy is an important metric invariant, and it is related to volume entropy rigidity question in Riemannian geometry \cite{BCG}. See also \cite{Lim} for a metric graph analogue and \cite{McMullen} for a related question.
The pair of topological entropy together with the ergodic period is a complete invariant for the equivalence relation of almost topological conjugacy, in the setting of ergodically supported expansive maps with shadowing property, including Anosov maps \cite{Sun}. As the volume entropy is equal to the topological entropy of the time one map of the geodesic flow, which is Anosov, together with the ergodic period, the pair is a complete invariant for the topological conjugacy. Volume entropy is also related to the first Betti number and systole of weighted graphs \cite{Balacheff}.

Recently, there have been attempts to use the volume entropy as an invariant to distinguish brain networks of certain patients \cite{LKKHLL}. One way to obtain a local invariant is to remove a vertex and all the edges emanating from it.

Another possible approach to use the entropy for networks is to use ``Persistent volume entropy" by introducing a parameter $\varepsilon$ similar to the parameter in persistent homology \cite{Carlsson}. Suppose that we are given a metric graph $\cG$. For given $\epsilon>0$, delete all the edges of length great than $\epsilon$ to obtain a graph $\cG_\varepsilon$ and consider the volume entropy of the graph $\cG_\varepsilon$. Since the graph $\cG$ is finite, the set of edge lengths $\{ \varepsilon_1 < \cdots < \varepsilon_m \}$ is finite and it is exactly the set of $\varepsilon>0$ such that the graph $\cG_{\varepsilon_i}$ strictly contains $\cG_{\varepsilon}$ for any small enough $\epsilon < \epsilon_i$.

In the second part of section 2, we suggest an algorithm of calculating entropy by specifying $\varepsilon$ which will be the threshold for two distinct ways of calculation: Newton's method using characterization in \cite{Lim} (Theorem 4) and a recursive formula for entropy when an edge is added.

Having such applications in mind, in this article, we investigate the change of volume entropy for two types of new graphs obtained from the original graph by either adding an edge or adding a vertex and all the edges emanating from it.

In the first part, we consider a graph $\cG$ and a new graph $\cG'$ obtained by adding an edge $e$. Our main result is the following.

\begin{thm} 
Let $\cG$ be a finite metric graph with two non-adjacent vertices $x$, $y$. Let $\cG'$ be a graph obtained by attaching an edge $e$ between the vertices $x,y$ of length $l_0$. Then the volume entropy $h=h_{\cG'}$ of $\cG'$ satisfies
$$e^{l_0h}=\sqrt{f_{xx}(h)f_{yy}(h)}+f_{xy}(h),$$ where $f_{xx},f_{yy},f_{xy}$ are the generating functions of $\cG$ starting and ending at $x$, starting and ending at $x$, and starting at $x$, ending at $y$, respectively.
\end{thm}

\begin{thm}
Let $\cG$ be a finite metric graph with two cycles whose ratioof lengths is Diophantine (see Definition \ref{def:Dio}) and with two non-adjacent vertices $x$, $y$. Let $\cG'$ be a graph obtained by attaching an edge $e$ between $x,y$ of length $l$. Let $h$ and $h'$ be the volume entropy of $\cG$ and $\cG'$, respectively. Then $h'=h+Ce^{-hl}+O(e^{-(1+\gamma)hl}))$ as $l\rightarrow \infty$ for some constant $C>0$ and any $\gamma<1$.
\end{thm}

Note that Diophantine condition  is generically satisfied by Khintchine Theorem. See Remark 2.10.

The proofs of the above theorems use various properties of generating functions.

\section{volume entropy change when adding an edge or a vertex}
Let us denote by $V\cG=\{ v_{1}, \cdots ,v_{k} \} $ the vertex set of $\cG$ and by $E\cG$ the set of oriented edges of $\cG$.

\subsection{Basic properties of entropy}
We denote an edge from $x$ to $y$ by $[x,y]$. We denote by $B(x,r)$ a ball of radius $r$ centered at $x$ in $\cG$. 
\begin{defn} 

\begin{enumerate}
\item By a \emph{path}, we mean a concatenation of adjacent edges of consistent orientation, or in other words, a metric path starting and ending at some vertices. For a \emph{path} $\bp$, we denote its length by $l(\bp)$.
\item For $x,y \in V\cG$, let us denote by $\cP_{xy}$ the set of paths without backtracking starting from $x$ and ending at $y$. Also denote $\displaystyle\bigcup_{y \in V\cG}\cP_{xy}$ by $\cP_{x}$.
\item Define the \emph{generating function of $\cG$ from $x$ (from $x$ to $y$}) by 
$$f_{x}(t)=\displaystyle\sum_{\bp\in{\cP_{x}}}e^{-l(\bp)t} \quad (f_{xy}(t)=\displaystyle\sum_{\bp\in{\cP_{xy}}}e^{-l(\bp)t},\; \mathrm{ resp}),$$ 
where $t \in \mathbb{R}.$
\item Define $N_{xy}(r)$ to be the number of paths in $\cG$ from $x$ to $y$ of length less than $r$ and $N_x(r)$ to be the number of paths in $\cG$ from $x$ of length less than $r$.
\item Define $vol(B(\tilde{x},r))$ to be the sum of length of all the edges in the metric ball $B(\tilde{x},r)$ of center $\tilde{x}$ and radius $r$ in the universal cover $\tilde{\cG}$ of $\cG$.
\end{enumerate}
\end{defn}
The folllowing lemma is a basic observation (see Lemma1 \cite{Lim} for a proof).
\begin{lem}\label{lem:0} For any vertex $x \in V\cG$, the volume entropy $h$ is the exponential growth rate of $N_x(r)$, i.e. $h=\displaystyle\lim_{r\rightarrow\infty}\frac{\log(N_x(r))}{r}$.
\end{lem}

Since the graph $\cG$ is finite, by summing over all the vertices of $\cG$, the volume entropy is also the exponential growth rate of the number of paths of length at most $r$.

There is a trivial upper bound of $h$: if $k+1$ is the maximum of the degree of vertices and $l$ is the minimum of the edge lengths, then the volume of a ball of radius $r$ is bounded above by $l \cdot k^{r/l+1}$, thus $h$ is bound above by $\frac{ \log k}{l}$. The next lemma gives another characterization of the volume entropy using the generating functions.
\begin{lem} Let $l$ be a length function on a set $\cP_{x}$. 
Let $f_{x}(t)=\sum_{\bp \in \cP_{x}} e^{-l(\bp) t}$ be the generating function of $\cG$ from $x$. Then, the volume entropy $h$ is the infimum of $t>0$ for which $f_{x}(t)$ converges.
\end{lem}

\begin{proof}
Let  $h_0$ be the infimum of $t$ for which $f_{x}(t)$ converges. Denote by $N(r)$ the number of $\bp \in \cP_x$ satisfying $l(\bp)<r$. 
For $t>h$, 
$$f_{x}(t)=\displaystyle\sum_{n=0}^{\infty}\displaystyle\sum_{\bp: n\leq l(\bp)<n+1} e^{-l(\bp) t}\leq \displaystyle\sum_{n=0}^{\infty} N_{x}(n+1) e^{-nt} = \displaystyle\sum_{n=0}^{\infty} e^{n(h-t+o(1))}.$$
It follows that $f_{x}(t)$ converges when $t>h$, thus we obtain $h\geq h_0$. 

On the other hand, if $h>h_0$, then choose $t$ such that $h>t>h_0$. For $r>0,$ 
$$f_{x}(t)e^{tr}=\displaystyle\sum_{\bp \in \cP_x} e^{(r-l(\bp))t}\geq \displaystyle\sum_{\{\bp \in \cP_x|l(\bp)<r\}} e^{(r-l(\bp))t}.$$ 
For $l(\bp)<r$, we have $e^{(r-l(\bp))t}>1$ thus $f_{x}(t)e^{tr} \geq N(r)$. 
Taking the log and letting $r$ go to infinity, we obtain
$ t \geq h,$
a contradiction. Thus $h \leq h_0$.
 \end{proof}

\begin{coro}\label{lem:1} If $\cG$ is connected, the volume entropy of $\cG$ is the infimum of $t>0$ for which the generating function $f_{xy}(t)$ of $\cG$ converges, for any $x,y \in V\cG$. \end{coro}
\begin{proof} Let $h_{xy}$ be the exponential growth rate of the number of paths from $x$ to $y$ of length at most $r$. Denoting by $N_{xy}(r)$ the number of paths from $x$ to $y$ of length at most $r$ and $h_{xy}=\displaystyle\lim_{r\rightarrow\infty}\frac{\log(N_{xy}(r))}{r}$.  Suppose that $y'$ is connected to $y$ and the distance between them is $l$.We have
$$N_{xy'} (r-l) \leq N_{xy} (r) \leq N_{xy'} (r+ l).$$
Again, by taking the exponential growth rate, we have $h_{xy} = h_{xy'}$.
\end{proof}

For a general graph $\cG$, it is immediate that the volume entropy is the maximum of the volume entropy of connected components.

\subsection{Adding an edge} In this subsection, we investigate how entropy changes when we add an edge to a given graph $\cG$.  
Let $\cG'$ be the graph obtained by attaching an edge $e=[x,y]$ of length $l_0$, i.e.
$$V\cG'=V\cG, E\cG'=E\cG\cup\{e, \overline{e}\},$$
where $\overline{e}=[y,x]$ in $\cG$.
We assume that two vertices $x$ and $y$ in $\cG$ are not adjacent.  

\begin{thm}\label{thm:1} Let $\cG$ be a finite metric graph with two non-adjacent vertices $x$, $y$. Let $\cG'$ be a graph obtained by attaching an edge $e$ between $x,y$ of length $l_0$. Then the volume entropy $h=h_{\cG'}$ of $\cG'$ satisfies
$$e^{l_0h}=\sqrt{f_{xx}(h)f_{yy}(h)}+f_{xy}(h),$$ where $f_{xx},f_{yy},f_{xy}$ are generating functions of $\cG$, from $x$ to $x$, from $y$ to $y$, and from $x$ to $y$, respectively.
\end{thm}

\begin{proof} 
Any path in $\cG'$ from $x$ to $y$ can be represented as a concatenation 
$$\bb_1 \be_1 \cdots \bb_{n-1} \be_n \bb_n $$
of paths $\bb_i$ in $\cG$ and the new edge $\be_i = e$ or $\eb$. 

Each path $\bb_i$ belongs to one of $\cP_{xx}, \cP_{xy}, \cP_{yx}, \cP_{yy}$ depending on $\be_{i-1}$ and $\be_i$. For example, $\bb_i \in \cP_{yx}$ if $\be_{i-1}=\be_i = e$ and $\bb_i \in \cP_{xx}$ if $\be_{i-1} =\eb, \be_i = e$. 

Note that if $\bb_i \in \cP_{xx}$ ($\bb_i \in \cP_{yy}$), then $\bb_{i+1} \in \cP_{yx} \cup \cP_{yy}$ ($\bb_{i+1} \in \cP_{xy} \cup \cP_{xx},$ resp.) since they are seperated by $e$ ($\eb$, resp.).
Therefore, 
one can subdivide any given path into a concatenation 
$$ \prod_{i=1}^m \bP_k =  \prod_{k=1}^{m} \bp_k e (\bq_{k{i_1}} e \cdots \bq_{k{i_k}} e) \br_k \eb (\bs_{k{j1}} \eb \cdots \bs_{k{j_k}} \eb) $$
of paths $\bP_k = \bp_k e (\bq_{k{i_1}} e \cdots \bq_{k{i_k}} e) \br_k \eb (\bs_{k{j_1}} \eb \cdots \bs_{k{j_k}} \eb),$ where we denote paths in $\cP_{xx}, \cP_{yx}, \cP_{yy}, \cP_{xy}$ by $\bp, \bq, \br, \bs$, respectively and $i_k, j_k$ are nonnegative. When $i_k=0$($j_k=0$), $\bq_0$($s_0$) is empty path and we have consecutive $e$'s.
Each path $\bP_k$ is determined by paths in $\cP_{xx}, \cP_{yx}, \cP_{yy}, \cP_{xy}$ and the value $e^{-l(\bP_k)t}$ of the path $\bP_k$ is a summand of
$$ f_{xx} e^{-l_0t} (f_{xy} e^{-l_0t})^{i_k} f_{yy}  e^{-l_0t} (f_{yx} e^{-l_0t})^{j_k} .$$
since $f_{xy} = f_{yx}$. Thus the original path $\prod_{k=1}^m \bP_k$ is a summand of the $m$-th term of the geometric series of common ratio

$\begin{aligned}
r
&= f_{xx}f_{yy } \left( e^{-l_0t}\sum_{i=0}^\infty (e^{-l_0t}f_{xy})^i \right)^2 = f_{xx}f_{yy } \left(\frac{e^{-l_0t}}{1-e^{-l_0t}f_{xy}} \right)^2\\
&=f_{xx}f_{yy}e^{-l_0t}\left(\displaystyle\sum_{i=0}^{\infty}(e^{-l_0t}f_{xy})^i\right)^2 .
\end{aligned}$

Therefore, the generating function of $\cG^\prime$ from $x$ to $y$ converges when $r<1$ and diverges when $r>1$. By Lemma~\ref{lem:1}, the volume entropy $h_{\cG'}$ is a zero of 
$$1=f_{xx}(t)f_{yy}(t) \left( \frac{e^{-l_0t}}{1-e^{-l_0t}f_{xy}(t)}\right)^2.$$
 Since $\frac{e^{-l_0t}}{1-e^{-l_0t}f_{xy}(t)}=e^{-l_0t}\displaystyle\sum_{i=0}^\infty (e^{-l_0t}f_{xy})^i$ is positive,\\ $$e^{l_0h_{\cG'}}=\sqrt{f_{xx}(h_{\cG'})f_{yy}(h_{\cG'})}+f_{xy}(h_{\cG'})$$ holds. \end{proof}

\begin{thm}\label{thm:2} Let $f$ be $f_{x}$ or $f_{xy}$. Then  $\frac{f(t)}{t}$ is the Laplace transform of $N(r)$ ($N_{x}(r)$ or $N_{xy}(r)$, respectively), i.e. for $t>0$, $$f(t)=t\int_0^\infty N(r)e^{-tr}dr$$.
\end{thm}

\begin{proof}
Note that $N(r)$ is a non-decreasing function and $N(0)=0$. For any $\delta>0$,

$\begin{aligned}
f(t)
&=\displaystyle\sum_{\bp\in{\cP}}e^{-l(\bp)t}
=\displaystyle\sum_{n=0}^{\infty}\displaystyle\sum_{n\delta< l(\bp) \leq (n+1)\delta}e^{-l(\bp)t}\\
&\leq\displaystyle\sum_{n=0}^{\infty} \{N((n+1)\delta)-N(n\delta)\}e^{-n\delta t}
=\displaystyle\sum_{n=1}^{\infty} N(n\delta) (e^{-(n-1)\delta t}-e^{-n\delta t})\\
&=\displaystyle\sum_{n=1}^{\infty} N(n\delta)e^{-n\delta t}(\delta t+O(\delta ^2t^2))
\end{aligned}$
\\The last summation converges to $t\int_0^\infty N(r)e^{-tr}dr$ when $\delta \rightarrow 0$. We obtain\\ $f(t)\leq t\int_0^\infty N(r)e^{-tr}dr$. Replacing $n\delta$ by $(n+1)\delta$ in the second line above, $f(t)\geq t\int_0^\infty N(r)e^{-tr}dr$ holds.
\end{proof}

\begin{defn}\label{def:Dio}
A real number $x$ is \emph{Diophantine} if there exist $\alpha , \beta >0$ such that $|x-\frac{p}{q}|<\alpha q^{-\beta}$ for all $p,q\in \bZ$ with $q>0$. We will call $\cG$ \emph{Diophantine} if there are has two cycles whose ratio of lengths is Diophantine.
\end{defn}

Broise-Alamichel, Parkkonen and Paulin \cite{BPP} showed that if $\cG$ is Diophantine, then there exists $C>0$ such that for every $n\in\bN$, as $r\rightarrow +\infty$,
\begin{equation}\label{eqn:5}
N(r)=Ce^{hr}(1+O(r^{-n}))
\end{equation}
We will compute the asymptotic behavior of volume entropy using this result. 

\begin{thm}\label{thm:3}
Assume $\cG$ is Diophantine. Let $f$ be either $f_x$ or $f_{xy}$. Then
$$f(t)=\frac{Ct}{t-h}(1+O((t-h)^\gamma))$$ 
as $t\rightarrow h^{+}$ for any $0<\gamma<1$. Here, $C$ is a constant that depends only on the graph $\cG$.
\end{thm}
\begin{proof}
For arbitrary $0<\gamma<1$, let $\alpha=\frac{1}{1-\gamma}-1>0$. There exists a function $m(r)$ such that $N(r)=Ce^{hr}+m(r)e^{hr}$, where $\lim_{r\rightarrow \infty}m(r)r^{\alpha}=0$ by \eqref{eqn:5}. For fixed $k$, there exists $\delta$ such that if $0<s<\delta$, then $m(r)<r^{-\alpha}$ for $r>s^{\gamma-1}$. From Theorem~\ref{thm:2},\\
$$f(t)=t\int_0^\infty N(r)e^{-tr}dr=\frac{Ct}{t-h}+t\int_0^\infty m(r)e^{-(t-h)} dr.$$
For the second part, assume that $t-h<\delta$. 
We have
\begin{align}
\int_0^\infty m(r)e^{-(t-h)r} dr&=\int_0^{(t-h)^{\gamma-1}} m(r)e^{-(t-h)r}dr+\int_{(t-h)^{\gamma-1}}^\infty m(r)e^{-(t-h)r}dr \\
&\leq {(t-h)^{\gamma-1}} \max_{r\geq 0}\{m(r)\}+ {(t-h)^{-\alpha(\gamma-1)}}\int_{(t-h)^{\gamma-1}}^\infty e^{-(t-h)r}dr\\
&\leq {(t-h)^{\gamma-1}} \max_{r\geq 0}\{m(r)\} +{(t-h)^{-\alpha(\gamma-1)}} (t-h)^{-1}\\
&= (\max_{r\geq 0}\{m(r)\}+1) (t-h)^{\gamma-1}.
\end{align}
Thus $|f(t)-\frac{Ct}{t-h}|\leq (\max_{r\geq 0}\{m(r)\}+1) (t-h)^{\gamma-1}$ for $t-h<\delta$.
It implies $f(t)=\frac{Ct}{t-h}(1+O((t-h)^\gamma))$ as $x\rightarrow h^{+}$ for each $0<\gamma<1$.
\end{proof}

\begin{coro}\label{coro:1}
Let $\cG$ be a finite metric Diophantine graph with two non-adjacent vertices $x$, $y$. Let $\cG'$ be a graph obtained by attaching an edge $e$ between $x,y$ of length $l$ to $\cG$. Let $h$ and $h'$ be the volume entropy of $\cG$ and $\cG'$, respectively. Then $$h'=h+Ce^{-hl}+O(e^{-(1+\gamma)hl}))$$ as $l\rightarrow \infty$ for some constant $C>0$ and any $\gamma<1$.
\end{coro}

\begin{proof}
Let $\cG$ be a Diophantine graph. From Theorem~\ref{thm:3} we can write 
\begin{align*}
f_{zw}(t)=\frac{C_{zw}t}{t-h}(1+O((t-h)^\gamma))
\end{align*}
for $(z,w)=(x,x), (x,y),$ and $(y,y)$.
Applying Theorem~\ref{thm:1} to these equations, we obtain 
$$e^{lh'}=\sqrt[]{f_{xx}(h')f_{yy}(h')}+f_{xy}(h')=\frac{C}{h'-h}\frac{h'}{h}(1+O((h'-h)^\gamma))$$ and
$$e^{lh}(h'-h)=Ce^{h-h'}\frac{h'}{h}(1+O((h'-h)^\gamma))=C+O((h'-h)^\gamma),$$
where $C=(\sqrt{C_{xx}C_{yy}}+C_{xy})h.$

Since $e^{lh}(h'-h)$ is bounded when $l\rightarrow \infty$, we have $e^{lh}(h'-h)=C+O(e^{-\gamma hl}),$
i.e. $$h'=h+Ce^{-hl}+O(e^{-(1+\gamma)hl})).$$ 
\end{proof}

\begin{rmk}[Generic Behavior]
We remark that the Diophantine condition for metric graphs is generic, since almost every real number is Diophantine by Khintchine Theorem or direct calculation.
Indeed, if we denote by $A$ the set of real numbers in $[0,1]$ which are not Diophantine, by definition, \\
$$A=\displaystyle\bigcap_{\alpha,\beta >0}\displaystyle\bigcup_{q=1}^{\infty}\displaystyle\bigcup_{p=0}^q \{x\in [0,1]||x-\frac{p}{q}|<\alpha q^{-\beta} \}$$
For all $\alpha,\beta>0$,
\begin{equation}
\begin{aligned}
\mu(A)
&\leq \displaystyle\sum_{q=1}^{\infty}\displaystyle\sum_{0\leq p\leq q}\mu(\{x\in [0,1]||x-\frac{p}{q}|<\alpha q^{-\beta} \})
\\&=\displaystyle\sum_{q=1}^{\infty}\displaystyle\sum_{0\leq p\leq q} 2\alpha q^{-\beta}
= 2\alpha\displaystyle\sum_{q=1}^{\infty} q^{-(\beta-1)}.
\end{aligned}
\end{equation}
Taking $\beta>2$ and sufficiently small $\alpha>0$, we obtain $\mu(A)=0$.
\end{rmk}
\subsection{Adding a vertex} In this subsection, we consider a new graph $\cG'$ obtained from $\cG$ by adding a vertex $v=v_0$ and edges $\{e_1, \dots, e_n\}$ emanating from $v_0$.   Denote
the terminal vertex of $e_i$  by $v_i$.
We first assume that the graph $\cG$ is connected and then will consider the general case at the end of the section.

The basic idea is similar to Theorem~\ref{thm:1}. We will split a cycle in $\cP_{vv}$ into primitive paths again. 
Any cycle $\bp$ in $\cP_{vv}$ is of the form
%
%
%
%
%
%
$$ \prod_{k=1}^m \bP_k  = \prod_{k=1}^m e_{i_k} \bq_{i_k j_k} \eb_{j_k}$$
of primitive paths $\bP_k = e_{i_k} \bq_{i_k j_k} \eb_{j_k}$ from $v$ to $v$, where $\bq_{i_k j_k}$ is a path in $\cP_{v_{i_k} v_{j_k}}.$ 

Let $F$ be an $n \times n$ matrix defined by
$$F_{i j} (t) = (1-\delta_{ij}) e^{-l_i t} f_{v_i v_j} (t) e^{-l_j t},$$
where $f_{v_i v_j}$ is the generating set of $\cG$ from $v_i$ to $v_j$. 
To exclude backtracking, there is an additional condition that
$ j_k \neq i_{k+1},$ which implies that the indices appearing in $\prod \bP_k$ is a summand of the matrix $A^m$ where $A_{ij} = 1- \delta_{ij}$. It follows that 
the value $e^{-l(\bp)}$ of each cycle is a summand of an entry of the matrix $F^m$.
Therefore we obtain the following theorem.
\begin{thm}\label{thm:4}
 The volume entropy $h_{\cG'}$ of $\cG'$ satisfies
$$||F(h_{\cG'})||=1,$$
where $|| \cdot ||$ denotes the spectral radius.
\end{thm}

Let $L,M$ be $(n \times n)$ matrices defined by
$$L_{ij}(t)=(1-\delta_{ij}) e^{-(l_i+l_j) t},  \qquad
M_{ij}(t)=f_{v_iv_j}(t).$$
The volume entropy $h_{\cG'}$ also satisfies
\begin{equation}\label{eqn:6}
1=||F(h_{\cG'})||=||L(h_{\cG'})||\cdot||M(h_{\cG'})||
\end{equation}
By an argument similar to Corollary~\ref{coro:1}, we have the following result.

\begin{thm}\label{thm:5}
Let $\cG$ be a finite metric Diophantine graph.
Let $h$ and $h'$ be the volume entropy of $\cG$ and $\cG'$, respectively. Then $$h'=h+C||L(h)||+O(||L(h)||^{1+\gamma})$$ as $l_{min}=\min \{l_{i}\}\rightarrow \infty$ for some constant $C>0$ and any $\gamma<1$.
\end{thm}

\begin{proof}
Note that $||L(t)||$ is continuous since $(1-\delta_{ij}) e^{-(l_i+l_j) t}$s, each components of $L(t)$, are continuous. By \ref{thm:3}, we can find constant $c_{ij}$s such that
$$f_{v_{i}v_{j}}=\frac{c_{ij}t}{t-h}(1+O((t-h)^\gamma)$$
as $t\rightarrow h^{+}$ for any $0<k<1$. Let $K$ be a matrix such that $K_{ij}=c_{ij}$. Then 
$$||M(t)||=\frac{t}{t-h}||K+O((t-h)^\gamma)||=\frac{t}{t-h}||K||(1+O((t-h)^\gamma))$$
holds. Applying equation \ref{eqn:6}, we obtain
$$1=||L(h')||\cdot||M(h')||=\frac{h'}{h'-h}||K||\cdot||L(h')||(1+O((h'-h)^\gamma)),$$
and 
$$\frac{h'-h}{||L(h)||}=C\frac{||L(h')||}{||L(h)||}\frac{h'}{h}(1+O((h'-h)^\gamma))=C+O((h'-h)^\gamma)$$
where $C=||K||h$ since $h'\rightarrow h$ and $||L(h')||\rightarrow||L(h)||$ as $l_{min}\rightarrow\infty$.
Since $\frac{h'-h}{||L(h)||}$ is bounded, we have $\frac{h'-h}{||L(h)||}=C+O(||L(h)||^\gamma)$.
\end{proof}

\begin{rmk}
The constant $C$ of Corollary \ref{coro:1}, \ref{eqn:5}, and Theorem \ref{thm:5} are all equal and depends on the constant of Theorem \ref{thm:3}. We have an upper bound of this constant from the proof of the Theorem \ref{thm:6}. The constant of equation \ref{eqn:5} and Theorem \ref{thm:3} have a upper bound $\frac{n-1}{n-2}\frac{\sum w_i}{\min \{ w_i \}}$. Thus the constants of Corollary \ref{coro:1} and Theorem \ref{thm:5} have upper bound $\frac{n-1}{n-2}\frac{\sum w_i}{\min \{ w_i \}}h$. $w_i$s are components of eigenvector of a matrix composed of generating functions. More details for this upper bound can be seen in the next section.
\end{rmk}

\section{Upper and lower bounds}
In this section, we provide an upper bound and a lower bound on the number $f_{xx}(h), f_{xy}(h), f_{yy}(h)$ that appear in Theorem~\ref{thm:1}. For the exact calculation of volume entropy, we need to count the number of paths or cycles without backtracking. However, the case with backtracking is much simpler, thus in order to convey the core idea of the proof, we first consider the case with backtracking which is of independent interest in relation to random walk.

\subsection{backtracking case}
Let us assume throughout this subsection that $\cG$ is connected and $|V\cG|\geq 3$.

 Let $N(r)$ be the number of cycles of length less than $r$. By Lemma~\ref{lem:0} and Corollary~\ref{lem:1}, the volume entropy $h$ of $\cG'$ satisfies $h=\displaystyle\lim_{r\rightarrow\infty}\frac{\log(N(r))}{r}$. 

For a vertex, $v\in \cG$, denote by $\cC_{v}$ the set of cycles with backtracking with initial vertex and terminal vertex both $v$. Denote by $\cD_{v}$ the set of primitive cycles in $\cC_v$, i.e. cycles which do not pass $v$ except at the initial and terminal vertices. Then the length spectrum of the $\cD_v$ i.e. set of the length of cycles in $\cD_v$ is a discrete set because $\cG$ is finite.  

Denote by $N_v(r)$ and $N_v'(r)$ the number of cycles in $\cC_v$ and $\cD_v$ of length less than $r$, respectively.
Denote by $h_\cC, h_\cD$ the exponential growth rate of $N_v(r), N_v'(r)$, respectively. By Corollary~\ref{lem:1}, $h_\cC$, ($h_\cD$) is the infimum of $t$ for which the generating function $$f(t)=\displaystyle\sum_{\bp\in{\cC_{v}}}e^{-l(\bp)t} \qquad (g(t)= \displaystyle\sum_{\bp\in{\cD_{v}}}e^{-l(\bp)t})$$ of $\cC_v$ ($\cD_v$, respectively) converges. Then we have a formal form
$$f(t)=g(t)+g(t)^2+g(t)^3+\cdots$$
By Corollary~\ref{lem:1}, $h_\cC$ satisfies 
$$g(h_\cC)= \displaystyle\sum_{\bp\in{\cD_{v}}}e^{-l(\bp)h_\cC}=1.$$
 Note that $|\cD_v|>1$ since $|V\cG|\geq 3$ and we have $h_{\cC}>0$. Since $g$ is continuous, there exists $\epsilon>0$ such that $g$ converges in $[h_\cC-\epsilon,h_\cC +\epsilon] $. Therefore $h_\cD < h_\cC$. Choose $h$ such that $h_\cD < h < h_\cC.$
 
 Also, for some constants $c^\prime>c>1$, choose $h^\prime$ such that $g(h^\prime)=c^\prime$. Then,  $N_v^\prime(r)=e^{(h_{\cD}+o(1))r}$ and $h>h^\prime>h_{\cD}$. Let $0<\epsilon<h^\prime-h_{\cD}$.
\\
\begin{thm} Let $\cG$ be a finite metric graph with $|V\cG|\geq 3$ and $v\in V\cG$. Denote by $N_v^b(r)$ the number of cycles with backtracking in $\cG$ with initial vertex $v$ and length less than $r$. Let $h$ be a exponential growth rate of $N_v^b(r)$. Then there exist $m,M,R>0$ which satisfy $$me^{hr}\leq N^{b}(r)\leq Me^{hr}$$ for $r>R$.
\end{thm}
\begin{proof}
Since $g(h')=c'$, we can find $R>0$ such that for $r>R$,
\begin{align}
&N_v^\prime(r)<e^{(h_{\cD}+\epsilon)r}
\\&c< \sum_{l(\bp) < r} e^{-h' l(\bp)}<c'.
\end{align}

From $g(h)=1$, for $r>R$, 
\begin{equation}\label{eqn:3}
\begin{aligned}
1-\sum_{\{ \bp \in \cD_v : l(\bp) < r\}} e^{-h l(\bp)}&=\sum_{\{ \bp \in \cD_v : l(\bp) \geq r\}} e^{-h l(\bp)}
\leq\displaystyle\sum_{n> r } \displaystyle\sum_{\{ \bp \in \cD_v : n-1\leq l(\bp)<n\}}e^{-hl(\bp)}\\
&\leq\displaystyle\sum_{n>r} N_v^\prime (n)e^{-h(n-1)}
<\displaystyle\sum_{n>r} e^{(h_{\cD} +\epsilon)n}e^{-h(n-1)}\\
&=\frac{e^{(h_{\cD} +\epsilon -h)(\left[r \right]+1)}}{1-e^{(h_{\cD} +\epsilon -h)}} e^{h_{\cD}+\epsilon}
<\frac{e^{(h_{\cD} +\epsilon -h)r}}{1-e^{(h_{\cD} +\epsilon -h)}} e^{h_{\cD}+\epsilon}\\
\end{aligned}
\end{equation}
For $R>0$ chosen above, we choose $m,m^\prime >0$ which satisfy the next two conditions:
\begin{equation}\label{eqn:4}
m+m^\prime <e^{-hR} \qquad \mathrm{and} \qquad
m\frac{e^{h_{\cD}+\epsilon}}{1-e^{(h_{\cD} +\epsilon -h)}}=(c -1)m^\prime
\end{equation}
\begin{lem}
$\forall r\geq 0, N_v(r)>me^{hr}+m^\prime e^{h^\prime r}-1$.
\end{lem}
\begin{proof} We use induction.
For $0\leq r\leq R$, from $h ' < h$ and the above condition (i),
$$me^{hr}+m^\prime e^{h^\prime r}<(m+m^\prime)e^{hR}<1\leq N_v(r)+1.$$
Now it remains to show that if  $N_v(r)>me^{hr}+m^\prime e^{h^\prime r}$  holds in $r\in \left[ 0,r_0\right]$, also it holds in $r\in \left[ 0,r_0+l_1\right]$, where $l_1$ is the minimum of the length of cycles in $\cD_v$.

For a cycle $\bp\in \cC_v$ such that $l(\bp)\leq r$, we can divide $\bp$ to $\bp_i\in \cD_v$ to satisfy $\bp=\bp_1 \cdots \bp_n,$ where each path $\bp_i$ have length less than $r$. For the rest of the proof, let us denote the set $\{ \bp \in \cD_v : l(\bp) < r \}$ by  $\{ l_1\leq l_2 \leq \cdots \leq l_{k} \}$. From counting each case of $l_i$ with multiplicity, we obtain 
\begin{equation}\label{eqn:2}
N_v(r)=k+\displaystyle\sum_{i=1}^k N_v(r-l_i)
\end{equation}

If $r_0 < r\leq r_0+l_1$, each $r-l_i$ is in $\left[ 0,r_0\right] $, thus by induction hypothesis, $$me^{h(r-l_i)}+m^\prime e^{h^\prime (r-l_i)} < N_v(r-l_i)+1.$$ Thus,

\begin{align*}
N_v(r)&=k+\displaystyle\sum_{i=1}^k N_v(r-l_i)
>\displaystyle\sum_{i=1}^k me^{h(r-l_i)}+\displaystyle\sum_{i=1}^k m^\prime e^{h^\prime(r-l_i)}
\\&=me^{hr}\displaystyle\sum_{i=1}^k e^{-hl_i}+m^\prime e^{h^\prime r}\displaystyle\sum_{i=1}^k e^{-h^\prime l_i}
\\&>me^{hr}\left(1-\frac{e^{h_{\cD}+\epsilon}e^{(h_{\cD} +\epsilon -h)r}}{1-e^{(h_{\cD} +\epsilon -h)}} \right)+cm^\prime e^{h^\prime r}
\\&=me^{hr}+m^\prime e^{h^\prime r}+(c-1)m^\prime e^{h^\prime r}-m\frac{e^h}{1-e^{(h_{\cD} +\epsilon -h)}} e^{(h_{\cD} +\epsilon)r}
\\&=me^{hr}+m^\prime e^{h^\prime r}+(c-1)m^\prime (e^{h^\prime r}-e^{(h_{\cD} +\epsilon)r})
>me^{hr}+m^\prime e^{h^\prime r}-1,
\end{align*}
where the last equality uses the second condition of Equation~\eqref{eqn:4}
\end{proof}
As for the upper bound, we claim that $\forall r\geq 0,$ $$N_v(r)\leq Me^{hr}-1,$$ where $M=max\{2,3e^{-hl_1}\}$. 
We again use the induction.

In the case of  $0\leq r< l_1$, $N_v(r)=0\leq  Me^{hr}-2$ since $M\geq 2$. In the case of $l_1\leq r< l_2$, $N_v(r)=1\leq  Me^{hr}-2$  since $M\geq 3e^{-hl_1}$. Thus, $N_v(r)\leq Me^{hr}-2$ holds for $r\in \left[ 0,l_2\right)$.

Now we will show that if $r_0\geq l_2$ and  $N_v(r)\leq Me^{hr}-1$  holds in $r\in \left[ 0,r_0\right)$, also it holds in $r\in \left[ 0,r_0+l_1\right)$. Let assume $r\in \left[r_0,r_0+l_1\right)$, then
\begin{align*}
N_v(r)&=k+\displaystyle\sum_{i=1}^k N_v(r-l_i)
\leq k+\displaystyle\sum_{i=1}^k (Me^{h(r-l_i)}-2)
\\&\leq Me^{hr}\displaystyle\sum_{i=1}^k e^{-hl_i} - k
\leq Me^{hr}-2
\end{align*}
Last inequality holds because $r\geq l_2$ implies $k\geq 2$. \end{proof}

\subsection{non-backtracking case}
Now, we come back to the more refined calculation of volume entropy and treat the non-backtracking case. Let us assume throughout the section that $\cG$ is connected and the number of cycles is more than one, so that $h_{\cG}>0$.

 Denote the edges in $\cG$ emanating from $v_0$ by $e_1, \cdots, e_n$ and 
the terminal vertex of $e_i$  by $v_i$, for $i =1, \cdots, n.$
We may assume $n\geq 3$ because removing vertices whose valency is less than 3 doesn't affect the volume entropy.

Denote by $\cC_{ij}$ the set of cycles starting with the edge $e_i$ and ending with the edge $e_j$.
Denote by $\cD_{ij}$ the set of primitive cycles in $\cC_{ij}$ which do not pass $v_0$ except in the beginning and at the end. Let us denote the set $\{ l(\bp)| \bp\in\cD_{ij}, l(\bp)<r \}$ by $\{l_1^{ij} \leq l_2^{ij} \leq \cdots \leq l_{N_{ij}}^{ij} \}$. Let $l_{min}=\min\{l_1^{ij}|1\leq i,j \leq n \}$ and $l_{max}=\max \{l_1^{ij}|1\leq i,j \leq n\}$. Let $f_{ij}(t)=\displaystyle\sum_{p\in\cC_{ij}}e^{-l(p)t}$ and $g_{ij}(t)=\displaystyle\sum_{p\in\cD_{ij}}e^{-l(p)t}$. Note that $f_{v_0}(t)=\displaystyle\sum_{1\leq i,j \leq n}f_{ij}(t)$ and $g{v_0}(t)=\displaystyle\sum_{1\leq i,j \leq n}g_{ij}(t)$.

Let $A(t)=(a_{ij}(t))_{1\leq i,j \leq n}$ be the $n$ by $n$ matrix such that $a_{ij}(t)=\displaystyle\sum_{k\neq j}g_{ik}(t)$. Note that $(A^m(t))_{ij}$ is the generating function of cycles starting $e_i$ and ending without $e_j$, passing $v_0$ $m$ times. Thus $\displaystyle\sum_{k\neq j}f_{ik}(t)=(\displaystyle\sum_{m=1}^{\infty}A^m(t))_{ij}$.
Since $h^{\cG}$ is a infimum of $t$ for which $f(t)$ converges, the spectral radius $||A(t)||$ must be $1$ at $t=h^{\cG}$. By Perron-Frobenius Theorem, there exists a positive vector $w=(w_1, w_2, \cdots ,w_n)^t$ such that $Aw=w$ and $w_i>0$ for $i=1,\cdots,n$.

Denote by $N_{ij}(r)$ and $N_{ij}'(r)$ the number of cycles in $\cC_{ij}$ and $\cD_{ij}$ of length less than $r$, respectively. Let $N_i(r)=\displaystyle\sum_{j=1}^n N_{ij}(r)$. By similar argument we used to obtain Equation~\ref{eqn:2}, we also obtain
$$N_i(r)=\displaystyle\sum_{j=1}^n\displaystyle\sum_{l_m^{ij}<r}(1+\displaystyle\sum_{k\neq j}N_k(r-l_m^{ij})).$$

Denote by $h_{\cD}$ the exponential growth rate of $N_{ij}'(r)$ for each $i,j$. The value $h_{\cD}$ is independent of $i,j$, since these are both equal to the volume entropy of the graph $\cG^*$, which is the graph obtained by removing the vertex $v_0$ and the edges $e_1,e_2,\cdots , e_n$ from $\cG$. Since each $g_{ij}$ is continuous and converges at $h$, there exists $\epsilon$ such that $g$ converges in $[h-\epsilon,h+\epsilon] $. Therefore $h_\cD < h$.

 For any constant $c>1$, choose $h_{\cD}<h^\prime<h$ such that $g_{ij}(h^\prime)>c^\prime g_{ij}(h)$ for all $1\leq i,j \leq n$.
\\Then,  $N_{ij}^\prime(r)=e^{(h_{\cD}+o(1))r}$ and $h>h^\prime>h_{\cD}$. Let $0<\epsilon<h^\prime-h_{\cD}$.

Since $g(h')>cg_{ij}(h)$, we can find $R>0$ such that for $r>R$,
\begin{align}
&N_{ij}^\prime(r)<e^{(h_{\cD}+\epsilon)r} & \forall 1\leq i,j \leq n.
\\& \displaystyle\sum_{\{ \bp \in \cD_{ij}: l(\bp) < r-l_{max}\}} e^{-h' l(\bp)}>cg_{ij}(h) & \forall 1\leq i,j \leq n.
\end{align}

As equation ~\eqref{eqn:3}, we can find a constant $a>0$ such that\\ $\displaystyle\sum_{\{ \bp \in \cD_{ij} : l(\bp) \geq r-l_{max}\}} e^{-h l(\bp)}\leq ae^{(h_{\cD} +\epsilon -h)r}$ holds for $r>R$.

For $R>0$ chosen above, we choose $m,m^\prime >0$ which satisfy the next two conditions:
\begin{equation}
(m+m^\prime)\max_{1\leq i\leq n}{w_i} <e^{-hR} \qquad \mathrm{and} \qquad
am\displaystyle\sum_{j=1}^n\displaystyle\sum_{k\neq j}w_k=(c -1)m^\prime \min_{1\leq i\leq n}w_i
\end{equation}\label{eqn:1}

\begin{lem}
$\forall r\geq l_{max}, N_i(r)>mw_ie^{hr}+m^\prime w_ie^{h^\prime r} (1\leq i\leq n)$.
\end{lem}
\begin{proof} We use induction.
For $l_{max}\leq r\leq R$, from $h ' < h$ and the above condition (i),
$$mw_ie^{hr}+m^\prime w_ie^{h^\prime r}<(m+m^\prime)(\max_{1\leq i\leq n}w_i)e^{hR}<1\leq N_i(r).$$
Now it remains to show that if  $N_i(r)>mv_ie^{hr}+m^\prime v_ie^{h^\prime r}$  holds in\\ $r\in \left[ l_{max},r_0\right]$, also it holds in $r\in \left[ l_{max},r_0+l_{min}\right]$. Thus,
\begin{align*} 
N_i(r)&>\displaystyle\sum_{j=1}^n\displaystyle\sum_{l_m^{ij}<r-l_{max}}(1+\displaystyle\sum_{k\neq j}(mw_ke^{h(r-l_m^{ij})}+m^\prime w_ke^{h^\prime(r-l_m^{ij})}))
\\&\geq me^{hr}\displaystyle\sum_{j=1}^n\displaystyle\sum_{k\neq j}w_kg_{ij}(h)-me^{hr}\displaystyle\sum_{j=1}^n\displaystyle\sum_{k\neq j}w_k\displaystyle\sum_{l_m^{ij}\geq r-l_{max}}e^{-hl_m^{ij}}
\\&+m^\prime e^{h^\prime r}\displaystyle\sum_{j=1}^n\displaystyle\sum_{k\neq j}w_k\displaystyle\sum_{l_m^{ij}<r-l_{max}}e^{-h^\prime l_m^{ij}}
\\&> mw_ie^{hr}-ame^{(h_{\cD}+\epsilon)r}\displaystyle\sum_{j=1}^n\displaystyle\sum_{k\neq j}w_k+m^\prime e^{h^\prime r}\displaystyle\sum_{j=1}^n\displaystyle\sum_{k\neq j}cw_kg_{ij}(h)
\end{align*} Thus
\begin{align*}
N_i(r) &= mw_ie^{hr}-ame^{(h_{\cD}+\epsilon)r}\displaystyle\sum_{j=1}^n\displaystyle\sum_{k\neq j}w_k+cm^\prime w_i e^{h^\prime r}
\\&>mw_ie^{hr}-ame^{h^\prime r}\displaystyle\sum_{j=1}^n\displaystyle\sum_{k\neq j}w_k+cm^\prime w_i e^{h^\prime r}
\\&=mw_ie^{hr}-(c-1)m^\prime w_i e^{h^\prime r}+cm^\prime w_i e^{h^\prime r}=mw_ie^{hr}+m^{\prime}w_ie^{h^\prime r}.
\end{align*}

\end{proof}
By lemma, there exists $m>0$ such that $N(r)>me^{hr}$ for $r\geq l_{max}$. This statement also holds for $r\geq l_{min}$ since $N(r)\geq 1$ for $r\geq l_{min}$.

As for the upper bound, we claim that $\forall r\geq 0,$ $$N_i(r)\leq Mw_ie^{hr}-\frac{1}{n-2}, \qquad \forall 1\leq i \leq n.$$ 
We again use the induction.
There exists some constant $M>0$ satisfying the above inequality for $r\in \left[0,l_i \right]$ for each $i$. Indeed, take $M=\frac{n-1}{n-2}\frac{1}{\displaystyle\min_{1\leq i\leq n}\{w_i\}}.$

 Now we will show that if $r_0 \geq l_{i}$ and $N_i(r)\leq Mw_ie^{hr}-\frac{1}{n-2}$ for all $1\leq i \leq n$ and $r\in \left[0,r_0 \right]$, also it holds in $r\in \left[0,r_0+l_{min} \right]$. Let assume $r\in \left[r_0,r_0+l_{min} \right]$, then by induction hypothesis,

\begin{align}
N_i(r)&\leq\displaystyle\sum_{j=1}^n\displaystyle\sum_{l_m^{ij}<r}(1+\displaystyle\sum_{k\neq j}(Mw_ke^{h(r-l_m^{ij})}-\frac{1}{n-2}))
\\&=Me^{hr}\displaystyle\sum_{j=1}^n\displaystyle\sum_{k\neq j}w_k\displaystyle\sum_{l_m^{ij}<r}e^{-hl_m^{ij}}-\displaystyle\sum_{j=1}^n\displaystyle\sum_{l_m^{ij}<r}\frac{1}{n-2}
\\&\leq Me^{hr}\displaystyle\sum_{j=1}^n\displaystyle\sum_{k\neq j}v_kg_{ij}(h)-\frac{1}{n-2}
=Me^{hr}\displaystyle\sum_{k=1}^n w_k\displaystyle\sum_{j\neq k}g_{ij}(h)-\frac{1}{n-2}
\\&=Me^{hr}\displaystyle\sum_{k=1}^n w_ka_{ik}(h)-\frac{1}{n-2}
=Mw_ie^{hr}-\frac{1}{n-2}
\end{align}

$\displaystyle\sum_{j=1}^n\displaystyle\sum_{l_m^{ij}<r}\frac{1}{n-2}\leq \frac{1}{n-2}$ holds since $r>l_{max}$ and $\displaystyle\sum_{k=1}^n w_ka_{ik}(h)=w_i$ holds since $Aw=w$.

 By induction, $N_i(r)\leq Mw_ie^{hr}-\frac{1}{n-2}$ holds for $r>0$. Thus upper bound $N(r)=\displaystyle\sum_{i=1}^n N_i(r)\leq M\displaystyle\sum_{i=1}^nw_i e^{hr}=\frac{n-1}{n-2}\frac{\sum w_i}{\min \{w_i\}}e^{hr}$ holds.
 
\begin{thm}\label{thm:6}
Let $\cG$ be a finite connected metric graph with more than one cycle. For $v\in V\cG$, denote by $N_v(r)$ the number of cycles of length less than $r$ which has no backtracking in $\cG$ and whose initial and terminal vertex are both $v$. Let $h$ be the volume entropy of the graph $\cG$. Let $l_{min}$ the length of the shortest cycle in $\cG$. Then there exist $m,M>0$ such that for $r\geq l_{min}$, 
$$me^{hr}\leq N_v(r)\leq Me^{hr}.$$ 
\end{thm}

\end{document}